\documentclass[12pt,a4paper,reqno]{amsart}
\usepackage{amssymb,bm}
\usepackage{latexsym}
\usepackage{exscale}
\usepackage{mathrsfs}
\usepackage{color}

\newcommand{\R}{\mathbb R}

\newcommand{\N}{\mathbb N}

\newcommand{\sG}{\mathscr G}
\newcommand{\sB}{\mathscr B}

\newcommand{\be}{{\mathbf e}}
\newcommand{\ben}{{\be^*_n}}

\newcommand {\SN} {{\mathbb N}}

\newcommand {\SX} {{\mathbb X}}

\newcommand {\Ga} {{\Gamma}}

\numberwithin{equation}{section}
\newtheorem{theorem}{Theorem}[section]
\newtheorem{lemma}[theorem]{Lemma}
\newtheorem{defi}[theorem]{Definition}
\newtheorem{corollary}[theorem]{Corollary}
\newtheorem{Remark}[theorem]{Remark}

\newtheorem{example}[theorem]{Example}

\newcommand{\Ba}[1]{\begin{array}{#1}}
\newcommand{\Ea}{\end{array}}
\newcommand{\Bd}{\begin{description}}
\newcommand{\Ed}{\end{description}}
\newcommand{\Be}{\begin{equation}}
\newcommand{\Ee}{\end{equation}}
\newcommand{\Bea}{\begin{eqnarray}}
\newcommand{\Eea}{\end{eqnarray}}
\newcommand{\Beas}{\begin{eqnarray*}}
\newcommand{\Eeas}{\end{eqnarray*}}
\newcommand{\Benu}{\begin{enumerate}}
\newcommand{\Eenu}{\end{enumerate}}
\newcommand{\Bi}{\begin{itemize}}
\newcommand{\Ei}{\end{itemize}}

\newcommand{\BR}{\begin{Remark} \em}
\newcommand{\ER}{\end{Remark}}
\newcommand{\BE}{\begin{example} \em}
\newcommand{\EE}{\end{example}}

\newcounter{reg}
\begin{document}

\title[\tiny Approximation with respect to polynomials with constant coefficients]{The best m-term approximation with respect to polynomials with constant coefficients}%
\author{Pablo M. Bern\'a}
\address{Pablo M. Bern\'a
\\
Instituto Universitario de Matem\'atica Pura y Aplicada
\\
Universitat Polit\`ecnica de Val\`encia
\\
46022 Valencia, Spain} \email{pmbl1991@gmail.com}
\author{\'Oscar Blasco}
\address{\'Oscar Blasco
\\
Departamento de An\'alisis Matem\'atico
\\
Universidad de Valencia, Campus de Burjassot
\\
46100 Valencia, Spain} \email{oscar.blasco@uv.es}

\thanks{The first author is partially supported by GVA PROMETEOII/2013/013 and 19368/PI/14 (\textit{Fundaci\'on S\'eneca}, Regi\'on de Murcia, Spain).  The second author is partially supported by  MTM2014-53009-P (MINECO, Spain).}


\begin{abstract}
In this paper we show that that greedy bases can be defined as those where the error term using $m$-greedy approximant is uniformly bounded by the best $m$-term approximation  with respect to polynomials with constant coefficients in the context of the weak greedy algorithm and weights.
\end{abstract}

\maketitle

\section{Introduction }
Let $(\SX,\Vert \cdot \Vert)$ be an infinite-dimensional real Banach space and let $\sB = (e_n)_{n=1}^\infty$ be a normalized Schauder basis of $\SX$ with biorthogonal functionals $(e_n^*)_{n=1}^\infty$. Throughout the paper, for each finite set $A\subset \SN$ we write $|A|$ for the cardinal of the set $A$, $1_A=\sum_{j\in A} e_j$ and $P_A(x)=\sum_{n\in A}e_n^*(x) e_n$. Given a collection of signs $(\eta_j)_{j\in A}\in\lbrace\pm 1\rbrace$ with $|A|<\infty$, we write $1_{\eta A} = \sum_{n\in A}\eta_j e_j\in \SX$ and we use the notation $[1_{\eta A}]$ and $[e_n, n\in A]$ for the one-dimensional subspace  and the $|A|-$dimensional subspace generated by generated by $1_{\eta A}$ and by $\lbrace e_n, n\in A\rbrace$ respectively. For each $x\in\SX$ and $m\in \SN$, S.V. Konyagin and V.N. Temlyakov defined in \cite{VT} the \textbf{$m$-th greedy approximant} of $x$ by
$$\mathcal{G}_m(x) = \sum_{j=1}^m e_{\rho(j)}^*(x)e_{\rho(j)},$$
where $\rho$ is a greedy ordering, that is $\rho : \SN \longrightarrow \SN$ is a permutation such that $supp(x) = \lbrace n: e_n^*(x)\neq 0\rbrace \subseteq \rho(\SN)$ and $\vert e_{\rho(j)}^*(x)\vert \geq \vert e_{\rho(i)}^*(x)\vert$ for $j\leq i$. The collection $(\mathcal{G}_m)_{m=1}^\infty$ is called the \textbf{Thresholding Greedy Algorithm} (TGA).

This algorithm is usually a good candidate to  obtain  the \textbf{best m-term approximation}  with regard to $\sB$, defined by
$$ \sigma_m(x,\sB)_\SX =\sigma_m(x) := \inf\lbrace d(x,[e_n, n\in A]) : A\subset \SN, \vert A\vert = m\rbrace.$$

The bases satisfying
\begin{equation} \label{old}\Vert x-\mathcal{G}_m(x)\Vert \leq C\sigma_m(x),\;\; \forall x\in\SX, \forall m\in\SN,\end{equation}
where $C$ is an absolute constant are called \textbf{greedy bases} (see \cite{VT}).

The first characterization of greedy bases was given by S.V. Konyagin and V. N. Temlyakov in \cite{VT} who established that a basis is greedy if and only if it is unconditional and democratic (where a basis is said to be democratic if there exists $C>0$ so that $\|1_A\|\le C \|1_B\|$ for any pair of finite sets $A$ and $B$ with $|A|=|B|$). 

Let us also recall two possible extensions of the greedy algorithm and the greedy basis. The first one consists in taking  the $m$ terms with near-biggest coefficients and  generating the Weak Greedy Algorithm (WGA) introduced by V.N. Temlyakov in \cite{T}.
 For each $t\in (0,1]$,
 a finite set $\Ga\subset\SN$ is called a $t$-greedy set for $x\in\SX$, for short $\Ga\in\sG(x,t)$, if
\[
\min_{n\in \Ga}|\ben(x)|\,\geq\,t\max_{n\notin\Ga}|\ben(x)|,
\]
and write $\Ga\in\sG(x,t,N)$ if in addition $|\Ga|=N$. A \textbf{$t$-greedy operator of order $N$} is a mapping  $G^t:\SX\to\SX$ such that
\[
G^t(x)=\sum_{n\in \Ga_x}\ben(x)\be_n, \quad \mbox{for some }\Ga_x\in \sG(x,t,N).
\]
   A basis is called \textbf{$t$-greedy} if there exists $C(t)>0$ such that
\begin{eqnarray}
\Vert x-G^t(x)\Vert \leq C(t)\sigma_m(x)\; \forall x\in\SX, \forall m\in\SN, \forall G^t\in \sG(x,t,m).
\end{eqnarray}

It was shown that a basis is $t$-greedy for some $0<t\le 1$ if and only if it is $t$-greedy for all $0<t\le 1$. From the proof it  follows that
greedy basis are also $t$-greedy basis with constant $C(t)= O(1/t)$ as $t\to 0$.

The second one consists in replacing $|A|$ by $w(A)=\sum_{n\in A} w_n$ and it  was considered by G. Kerkyacharian, D. Picard and V.N. Temlyakov in \cite{KPT} (see also \cite[Definition 16]{Tem}). Given a weight sequence $\omega = \lbrace \omega_n\rbrace_{n=1}^\infty, \omega_n >0$ and a positive real number $\delta>0$, they defined
$$\sigma_\delta ^\omega (x) = \inf \lbrace d(x,[e_n, n\in A]) : A\subset \SN, \omega(A)\leq \delta\rbrace$$
where $\omega(A) := \sum_{n\in A}\omega_n$, with $A\subset\mathbb{N}$.
 They called  \textbf{weight-greedy bases} ($\omega$- greedy bases) to those bases satisfying
\begin{eqnarray}\label{wt}
\Vert x-\mathcal{G}_m(x)\Vert \leq C \sigma_{\omega(A_m)}^\omega (x),\; \forall x\in\SX, \forall m\in\SN,
\end{eqnarray}
where $C>0$ is an absolute constant and $A_m = supp(\mathcal{G}_m(x))$. Moreover, they proved in \cite{KPT} that $\sB$ is a $\omega$- greedy basis if and only if it is unconditional and $w$-democratic (where a basis is $w$-democratic whenever there exists $C>0$ so that $\|1_A\|\le C \|1_B\|$ for any pair of finite sets $A$ and $B$ with $w(A)\le w(B)$). 
This generalization was motivated by the work of A. Cohen, R.A. DeVore and R. Hochmuth in \cite{CDH} where the basis was indexed by dyadic intervals and  $w_\alpha (\Lambda)=\sum_{I\in \Lambda}|I|^\alpha$. Later in 2013, similar considerations were considered by  E. Hern\'andez and D. Vera  to prove some inclusions of approximation spaces (see \cite{HV}).

Let us summarize and use the following combined definition.
\begin{defi} Let $\sB$ be a normalized Schauder basis in $\mathbb{X}$, $0<t\le 1$ and weight sequence $\omega = \lbrace \omega_n\rbrace_{n=1}^\infty$ with $\omega_n >0$.
We say that $\sB$ is \textbf{$(t,\omega)$-greedy} if there exists $C(t)>0$ such that
\begin{equation}\label{g}
\Vert x-G^t(x)\Vert \leq C(t)\sigma^w_{m(t)}(x)\; \forall x\in\SX, \forall m\in\SN, \forall G^t\in \sG(x,t,m)
\end{equation}
where $A_m(t)=supp (G^t(x))$ and $m(t)=w(A_m(t))$.
\end{defi}

 The authors  introduced (see \cite{BB}) the best $m$-term approximation with respect to polynomials with constant coefficients as follows:
$$\mathcal{D}^*_m(x) := \inf \lbrace d(x,[1_{\eta A}]) : A\subset \SN, (\eta_n)\in \{\pm 1\}, \vert A\vert = m\rbrace.$$
Obviously, $\sigma_m(x)\leq \mathcal{D}^*_m(x)$ but, while $\sigma_m(x)\to 0$ as $m\to\infty$ it was shown that for orthonormal bases in Hilbert spaces we have $\mathcal{D}^*_m(x)\to \|x\|$ as $m\to\infty$.
The following result establishes a new description of greedy bases using the best $m$-term approximation with respect to polynomials with constant coefficients.
\begin{theorem} (\cite[Theorem 3.6]{BB})
Let $\SX$ be a Banach space and $\sB$ a Schauder basis of $\SX$.

  (i) If there exists  $C>0$ such that
    \begin{equation}\label{new}\Vert x-\mathcal{G}_m(x)\Vert \leq C\mathcal{D}^*_m(x),\; \forall x\in \SX,\; \forall m\in \mathbb{N},\end{equation}
    then $\sB$ is $C$-suppression unconditional and $C$-symmetric for largest coefficients.

   (ii) If $\sB$ is $K_s$-suppression unconditional and $C_s$-symmetric for largest coefficients then
    $$\Vert x-\mathcal{G}_m(x)\Vert \leq (K_s C_s)\sigma_m(x),\; \forall x\in \SX,\; \forall m\in \mathbb{N}.$$

\end{theorem}

The concepts of suppression unconditional and symmetric for largest coefficients bases can be found in \cite{BB,AA2,AW,DKOSS,VT}. We recall here that a basis is \textbf{$K_s$-suppression unconditional} if the projection operator is uniformly bounded, that is to say
$$\Vert P_A(x)\Vert \leq K_s\Vert x\Vert,\; \forall x\in\SX,\forall A\subset \SN$$
 and $\sB$ is \textbf{$C_s$-symmetric for largest coefficients} if
$$\Vert x+t1_{\varepsilon A}\Vert \leq C_s\Vert x+t1_{\varepsilon' B}\Vert,$$
for any $\vert A\vert = \vert B\vert$, $A\cap B=\emptyset$, $supp(x) \cap (A\cup B) = \emptyset$, $(\varepsilon_j), (\varepsilon'_j) \in \lbrace \pm 1\rbrace$ and $t = \max\lbrace \vert e_n^*(x)\vert : n\in supp(x)\rbrace$.

In this note we shall give a direct proof of the equivalence between condition  (\ref{old}) and (\ref{new}) even in the setting of $(t,w)$-greedy basis.

 Let us now introduce our best $m$-term approximation with respect to polynomials with constant coefficients associated to a weight sequence and the basic property to be considered in the paper.
 \begin{defi} Let $\sB$ be a normalized Schauder basis in $\mathbb{X}$, $0<t\le 1$ and  a weight sequence $\omega = \lbrace \omega_n\rbrace_{n=1}^\infty$ with $\omega_n >0$. We denote by $$\mathcal{D}_{\delta}^\omega (x) := \inf \lbrace d(x,[1_{\eta A}]) : A\subset \SN, (\eta_n)\in \{\pm 1\}, \omega(A)\leq \delta\rbrace.$$
 The basis $\sB$ is said to $(t,w)$-greedy for polynomials with constant coefficients, denoted to have {\bf $(t,w)$-PCCG property}, if there exists $D(t)>0$ such that
\begin{equation}\label{ng}\Vert x-G^t(x)\Vert \leq D(t)\mathcal{D}_{m(t)}^\omega (x), \forall x\in\SX, \forall m\in\SN, \forall G^t\in \sG(x,t,m)\end{equation}
where $A_m(t)=supp(G^t(x))$  and $m(t)=\omega(A_m(t))$.

In the case $t=1$ and $w(A)=|A|$ we simply call it the {\bf PCCG property}.
\end{defi}
Of course $\sigma_\delta ^\omega(x) \leq \mathcal{D}_\delta^\omega(x)$ for all $\delta>0$, hence if the basis is $(t,\omega)$-greedy then (\ref{ng}) holds with the $D(t)=C(t)$. We now formulate our main result which produces a direct proof of the result in \cite{BB} and give  the extension to $t$-greedy and weighted greedy versions.
\begin{theorem} Let $\sB$ be a normalized Schauder basis in $\mathbb{X}$ and let $\omega = \lbrace \omega_n\rbrace_{n=1}^\infty$ be a weight sequence  with $\omega_n >0$ for all $n\in \N$. The following are equivalent:

(i) There exist $0<s\le 1$ such that $\sB$ has the $(s,w)$-PCCG property.

 (ii) $\sB$  is $(t,\omega)$-greedy for all $0<t\le 1$.
\end{theorem}

\begin{proof} Only the implication (i) $\Longrightarrow$ (ii) needs a proof.
Let us  assume that (\ref{ng}) holds for some $0<s\le 1$. Let $0<t\le 1$,  $x \in\SX$, $m\in \N$  and $G^t\in \sG(x,t,m)$. We write $G^t(x) = P_{A_m(t)}(x)$ with $A_m(t)\in \mathcal G(x,t,m)$. For each $\varepsilon >0$  we choose $z = \sum_{n\in B}e_n^*(x)e_n$ with  $\omega(B)\leq \omega(A_m(t))$ and $\Vert x-z\Vert \leq \sigma_{\omega(A_m)}^\omega (x) + \varepsilon$.

We  write
$$x- P_{A_m(t)}(x)= x- P_{A_m(t)\cup B} (x) +P_{B\setminus A_m(t)}(x).$$
Taking into account that $P_{B\setminus A_m(t)}(x)\in co(\lbrace S 1_{\eta (B\setminus A_m(t))} : \vert \eta_j\vert = 1\rbrace)$ for any $S\ge \underset{j\in B\setminus A_m(t)}{\max}\vert e_j^*(x)\vert$, it suffices to show that there exists $R\ge 1$ and $C(t)>0$ such that
\begin{equation}\label{final}
\|x- P_{A_m(t)\cup B} (x) + R\gamma  1_{\eta(B\setminus A_m(t))}\|\le C(t) \|x-z\|
\end{equation}
for any  choice of signs $(\eta_j)_{j\in B\setminus A_m(t)}$  where $\gamma = \underset{j\in B\setminus A_m(t)}{\max}\vert e_j^*(x)\vert$.

Let us assume first that $t\geq s$. We shall  show that
\begin{equation} \label{two}\Vert x- P_{(A_m(t)\cup B)}(x)+\frac{t}{s}\gamma 1_{\eta B\setminus A_m(t)}\Vert \le D(s)\Vert x- P_B(x)\Vert\end{equation}
for any choice of signs $(\eta_j)_{j\in B\setminus A_m(t)}$.

Given $(\eta_j)_{j\in B\setminus A_m(t)}$ we consider
$$y_{\eta}= x- P_B(x)+ \frac{t}{s}\gamma 1_{\eta(B\setminus A_m(t))}=\sum_{n\notin B} e_n^*(x)e_n+ \sum_{n\in B\setminus A_m(t)}\frac{t}{s}\gamma\eta_n e_n .$$
Note that $$\min_{n\in A_m(t)\setminus B}|e_n^*(y_\eta)|= \min_{n\in A_m(t)\setminus B}|e_n^*(x)|\ge \min_{n\in A_m(t)}|e_n^*(x)|$$ and $$s\max_{n\in (A_m(t)\setminus B)^c}|e_n^*(y_\eta)|=\max\{ s\max_{n\notin A_m(t)}|e_n^*(x)|, t\gamma\} .$$
Therefore, since $t\ge s$, we conclude that
$$\min_{n\in A_m(t)\setminus B}|e_n^*(y_\eta)|\ge s \max_{n\in (A_m(t)\setminus B)^c}|e_n^*(y_\eta)|.$$
 Hence $A_m(t)\setminus B \in \mathcal G(y_\eta, s, N)$ with $N = \vert A_m(t)\setminus B\vert$. We write $G^s(y_\eta) = P_{A_m(t)\setminus B}(x)$  and notice that
$$ y_\eta- G^s(y_\eta)= x- P_{A_m(t)\cup B}(x)+ \frac{t}{s}\gamma 1_{\eta(B\setminus A_m(t))}. $$

Since $\omega(B)\leq \omega(A_m(t))$ we have also that $\omega(B\setminus A_m(t))\leq \omega(A_m(t)\setminus B)$.
Hence for $N(s)=\omega(A_m(t)\setminus B)$ we conclude
\begin{eqnarray*}
\Vert x- P_{(A_m(t)\cup B)}(x)+\frac{t}{s}\gamma 1_{\eta B\setminus A_m(t)}\Vert &\leq& D(s)\mathcal{D}_{N(s)}^\omega (y_\eta)\\\nonumber
 &\leq& D(s) \Vert y_\eta-\frac{t}{s}\gamma1_{\eta B\setminus A_m(t)}\Vert\\\nonumber
&=&D(s)\Vert x- P_B(x)\Vert.
\end{eqnarray*}

 Now, let $y=x-z+ \mu 1_B$  for $\mu = s \, \underset{j\notin B}{\max}\vert e_j^*(x-z)\vert +\underset{j\in B}{\max}\vert e_j^*(x-z)\vert.$\newline

Then $$\min_{j\in B} |\mu + e_n^*(x-z)|\ge s\max_{j\notin B} |e_n^*(x-z)|,$$ which gives that $B\in \mathcal G(y,s, |B|)$ and we obtain
$G^s(y) = P_B(x-z)+\mu 1_{B}$. Hence
\begin{equation}\label{three}
\Vert x-P_B(x)\Vert = \Vert y-G^s(y)\Vert \le D(s)\Vert y - \mu 1_B\Vert= D(s)\Vert x-z\Vert.
\end{equation}

Therefore, by $\eqref{two}$ and $\eqref{three}$ we obtain
$$\Vert x- P_{(A_m(t)\cup B)}(x)+\frac{t}{s}\gamma 1_{\eta B\setminus A_m(t)}\Vert\le D(s)^2\|x-z\|.$$
Then, for $s\le t$ we obtain that $\sB$ is $(t,w)$-greedy with constant $C(t)\le D(s)^2$.

We  now consider the case $s>t$. We use the following estimates:
$$\Vert x- P_{(A_m(t)\cup B)}(x)+\gamma 1_{\eta B\setminus A_m(t)}\Vert\le \Vert x- P_{ B}(x)\Vert+\Vert P_{A_m(t)\setminus B}(x)\Vert+\gamma \Vert1_{\eta B\setminus A_m(t)}\Vert.$$
Arguing as above, using now
$$\tilde y_{\eta}= P_{A_m(t)\setminus B}(x)+ \frac{t}{s}\gamma 1_{\eta(B\setminus A_m(t))},$$
we conclude that  $\frac{t}{s}\gamma \Vert1_{\eta B\setminus A_m(t)}\Vert\le D(s)\|P_{A_m(t)\setminus B}(x)\|$.

The argument used to show (\ref{three}) gives  $ \|z- P_C z\|\le D(s) \|z\|$ for all $z\in \mathbb{X}$ and finite set $C$. Therefore
$$\Vert P_{A_m(t)\setminus B}(x)\Vert= \Vert P_{A_m(t)}(x-P_{B}x)\Vert\le (1+ D(s))\|x-P_Bx\|.$$
Putting all together we have
$$\Vert x- P_{(A_m(t)\cup B)}(x)+\gamma 1_{\eta B\setminus A_m(t)}\Vert\le (2+\frac{t+s}{t}D(s))\|x-P_Bx\|,$$
and therefore $\sB$ is $(t,w)$-greedy with constant $C(t)\le (2+\frac{t+s}{t}D(s))D(s).$
\end{proof}

\begin{corollary}
If $t=1$ and $\omega(A) = \vert A\vert$, then $\sB$ has the PCCG property if and only if $\sB$ is greedy.
\end{corollary}

\begin{corollary}
If $\omega(A) = \vert A\vert$, then $\sB$ has the $t$-PCCG property if and only if $\sB$ is $t$-greedy.
\end{corollary}

\section{A remark on the Haar system}

Throughout this section $|E|$ stands for the Lebesgue measure of a set in $[0,1]$, $\mathcal D$ for the family of dyadic intervals in $[0,1]$ and  $card(\Lambda)$ for the number of dyadic elements in $\Lambda$.
We denote by $\mathcal{H}:=\lbrace H_I\rbrace$ the Haar basis in $[0,1]$, that is to say
$$H_{[0,1]} (x) = 1\; \text{for}\; x\in [0,1),$$
and for  $I\in \mathcal D$ of the  form $I = [(j-1)2^{-n}, j2^{-n})$, $j=1,..,2^n$, $n= 0,1,...$ we have
\begin{displaymath}
H_{I}(x) = \left\{ \begin{array}{ll}
2^{n/2} & \mbox{if $x\in [(j-1)2^{-n}, (j-\frac{1}{2})2^{-n})$,} \\
-2^{n/2} & \mbox{if $x\in [(j-\frac{1}{2})2^{-n}, j2^{-n})$,} \\
0 & \mbox{otherwise.}
\end{array}
\right.
\end{displaymath}

We write $$c_I(f) := \langle f,H_I\rangle = \int_0^1 f(x)H_I(x)dx \hbox{ and } c_I(f,p):= \Vert c_I(f)H_I\Vert_p, \quad 1\le p<\infty.$$
 It is well known that $\mathcal H$  is an orthonormal basis in $L^2([0,1])$ and for $1<p<\infty$ we can use the Littlewood-Paley's Theorem which gives
\begin{equation}\label{lp} c_p \left\Vert \left( \sum_I \vert c_I(f,p)\frac{H_I}{\Vert H_I\Vert_p}\vert^2\right)^{1/2}\right\Vert_p \leq \Vert f\Vert_p \leq  C_p\left\Vert \left( \sum_I \vert c_I(f,p)\frac{H_I}{\Vert H_I\Vert_p}\vert^2\right)^{1/2}\right\Vert_p
\end{equation}
 to conclude that $(\frac{H_I}{\Vert H_I\Vert_p})_I$ is an unconditional basis in $L^p([0,1])$. Denoting $f<<_p g$ whenever $c_I(f,p)\le c_I(g,p)$ for all dyadic intervals $I$ we obtain from (\ref{lp}) the existence of a constant $K_p$ such that
 \begin{equation}\label{uncon}
 \|f\|_p\le K_p \|g\|_p \quad \forall f, g\in L^p([0,1]) \hbox{ with } f<<_p g,
 \end{equation}
 and also
 \begin{equation}\label{h}
\|P_\Lambda g\|\le K_p \|g\|_p \quad \forall g\in L^p \quad \forall \Lambda\subset \mathcal D.
\end{equation}
Regarding the greedyness of the Haar basis it was
V. N. Temlyakov the first one who proved (see \cite{T}) that the every wavelet basis $L_p$-equivalent to the Haar basis is $t$-greedy in $L_p([0,1])$ with $1<p<\infty$ for any $0<t\le 1$.

 Let $\omega:[0,1]\to \R^+$ be a measurable weight and, as usual, we denote $\omega(I)=\int_I \omega(x)dx$ and $m_I(\omega)=\frac{\omega(I)}{|I|}$ for any $I\in \mathcal D$. In the space $L^p(\omega)=L^p([0,1],\omega)$ we denote $\|f\|_{p, \omega}=(\int_0^1 |f(x)|^p\omega(x) dx)^{1/p}$ and $$ c_I(f,p,\omega):= \Vert c_I(f)H_I\Vert_{p,\omega}= |c_I(f)|\frac{\omega(I)^{1/p}}{|I|^{1/2}}. $$  Recall that $\omega$ is said to be a
dyadic $A_p$-weight (denoted $\omega \in A^{d}_p$) if
\begin{equation}
A^{d}_p(\omega)= \sup_{I\in \mathcal D} m_I(\omega) \Big( m_I(\omega^{-1/(p-1)})\Big)^{p-1}<\infty.
\end{equation}

As one may expect, Littlewood-Paley theory holds for weights in the dyadic $A_p$-class.
\begin{theorem} (see \cite{ABM, I} for the multidimensional case) If $\omega \in A^{d}_p$ then
\begin{eqnarray}\label{uncon}  \Vert f\Vert_{p,\omega} \approx\left\Vert \left( \sum_I \vert c_I(f,p,\omega)\frac{H_I}{\Vert H_I\Vert_{p,\omega}}\vert^2\right)^{1/2}\right\Vert_{p,\omega}.
\end{eqnarray}
In particular $(\frac{H_I}{\Vert H_I\Vert_{p,\omega}})_I$ is an unconditional basis in $L^p(\omega)$ for $1<p<\infty$.
\end{theorem}

 The greedyness of the Haar basis in $L^p(\omega)$ goes back to M. Izuki (see \cite{I, IS}) who showed that this holds for weights in the class $A_p^d$.
 We shall use the ideas in these papers to show that the Haar basis satisfies the PCCG property for certain spaces defined using the Littlewood-Paley theory.
  \begin{defi} Let $\omega:[0,1]\to \R^+$ be a measurable weight and $1\le p<\infty$. For each finite set of dyadic intervals $\Lambda$ we define
  $f_\Lambda=\sum_{I\in \Lambda} c_I(f)H_I=\sum_{I\in \Lambda} c_I(f,p,\omega)\frac{H_I}{\Vert H_I\Vert_{p,\omega}}$ and write $$\|f\|_{X^p(\omega)}= \left\Vert\left( \sum_{I\in \Lambda} \vert c_I(f,p,\omega)\frac{H_I}{\Vert H_I\Vert_{p,\omega}}\vert^2\right)^{1/2}\right\Vert_{p,\omega}.$$
  The closure of $span(f_\Lambda: card(\Lambda)<\infty)$ under this norm will be denoted $X^p(\omega)$.
  \end{defi}
  From the definition  $(\frac{H_I}{\Vert H_I\Vert_{p,\omega}})_I$ is an unconditional basis with constant 1 in $X^p(\omega)$ and due to (\ref{uncon}) $X^p(\omega)=L^p(\omega)$ whenever $\omega\in A_p^d$.  Our aim is to analyze conditions on the weight $\omega$ for the basis to be greedy. For such a purpose we do not need the weight to belong to $A_p^d$. In fact analyzing the proof in \cite{I, IS} one notices that only  the dyadic reverse doubling condition (see \cite[p. 141]{GCRF}) was used. Recall that a weight $\omega$ is said to satisfies {\bf the dyadic reverse doubling condition} if there exists $\delta<1$ such that
  \begin{equation}\label{dc}\omega(I')\le \delta \omega (I), \forall I,I'\in \mathcal D \hbox{ with } I'\subsetneq I.
  \end{equation}
  Let us introduce certain weaker conditions.
  \begin{defi} Let $\alpha>0$ and $\omega$ be a measurable weight. We shall say that  $\omega$  satisfies {\bf the dyadic reverse Carleson  condition} of order $\alpha$ with constant $C>0$ whenever
  \begin{equation}\label{cc}\sum_{I\in \mathcal D, J\subseteq I}\omega(I)^{-\alpha}\le C \omega(J)^{-\alpha}, \forall J\in \mathcal D .
  \end{equation}
  \end{defi}

  \begin{defi}  Let $\alpha>0$ and two sequences $(w_I)_{I\in \mathcal D}$ and  $(v_I)_{I\in \mathcal D}$ of positive real numbers. We say that the pair $\Big((w_I)_{I\in \mathcal D}, (v_I)_{I\in \mathcal D}\Big)$ satisfies  $\alpha-{\bf DRCC }$ with constant $C>0$ whenever
  \begin{equation}\label{cc1}\sum_{I\in \mathcal D, J\subseteq I}w_I^{-\alpha}\le C v_J^{-\alpha}, \forall J\in \mathcal D .
  \end{equation}
  \end{defi}

  \begin{Remark} \label{n} (i) If $\omega\in \cup_{p> 1}A_p^w$ then $\omega$ satisfies the dyadic reverse doubling condition (see \cite[p 141]{GCRF}).

  (ii) If $\omega$ satisfies the dyadic reverse doubling condition then $\omega$ satisfies the dyadic reverse Carleson condition of order $\alpha$ with constant $\frac{1}{1-\delta^\alpha}$ for any $\alpha>0$.

  Indeed, $$\sum_{J\subseteq I}\omega(I)^{-\alpha}\le \omega(J)^{-\alpha}+ \omega(J)^{-\alpha}\sum_{m=1}^{\infty} \delta^{m\alpha}\le \frac{1}{1-\delta^\alpha}\omega(J)^{-\alpha}.$$
  (iii) If $\omega$ satisfies the dyadic reverse Carleson condition of order $\alpha$ and $w_I=\omega(I)$ for each $I\in \mathcal D$ then
$\Big((w_I)_{I\in \mathcal D},(w_I)_{I\in \mathcal D}\Big)$ satisfies $\alpha$-{\bf DRCC }.
\end{Remark}

   We need the following lemmas, whose  proofs are essentially included in  \cite{CDH, I, IS}.

   \begin{lemma}  \label{1c} Let   $\omega$ be a weight and $(v_I)_{I\in \mathcal D}$ be a sequence of positive real numbers such that
 $\Big((v_I)_{I\in \mathcal D}, (\omega(I))_{I\in \mathcal D}\Big)$ satisfies $ 1$-{\bf DRCC } with constant $C$.
Then 
\begin{equation}\label{dem}
\left(\sum_{I\in \Lambda} \frac{\omega(I)}{v_I}\right)^{1/p}\le C\left\|\sum_{I\in \Lambda} \frac{H_I}{\Vert H_I\Vert_{p,\omega}}\right\|_{X^p(\omega)}, \forall 1\le p<\infty.
\end{equation}
\end{lemma}
\begin{proof} We first write
\begin{equation}\label{main} \left\|\sum_{I\in \Lambda} \frac{H_I}{\Vert H_I\Vert_{p,\omega}}\right\|_{X^p(\omega)}= \left(\int_0^1
(\sum_{I\in \Lambda} \omega(I)^{-2/p}\chi_I)^{p/2}\omega(x) dx\right)^{1/p}. \end{equation}
Let $I(x)$ denote the minimal dyadic interval in $\Lambda$ with regard to the inclusion relation that contains $x$. Now we use that $$\sum_{I\in \mathcal D, I(x)\subseteq I}v_I^{-1}\le C \omega(I(x))^{-1}$$ to conclude that
\begin{eqnarray*}
(\sum_{I\in \Lambda} \frac{\omega(I)}{v_I})^{1/p}&=& \Big(\sum_{I\in\Lambda} \int_{I}v_I^{-1}\omega(x)dx\Big)^{1/p}= \Big(\int_0^1(\sum_{I\in \Lambda}v_I^{-1}\chi_I(x))\omega(x)dx\Big)^{1/p}\\
&\le & C\Big(\int_0^1 \omega(I(x))^{-1}\omega(x)dx\Big)^{1/p}\le C\Big(\int_0^1(\sum_{I\in \Lambda}\omega(I)^{-2/p}\chi_I(x))^{p/2}\omega(x)dx\Big)^{1/p}\\
&=& C\|\sum_{I\in \Lambda} \frac{H_I}{\Vert H_I\Vert_p}\|_{X^p(\omega)}.
\end{eqnarray*}
The proof is complete.
\end{proof}
\begin{lemma} \label{2p}Let $1<p<\infty$,  $\omega$ be a weight and $(v_I)_{I\in \mathcal D}$ of positive real numbers. If
$\Big((\omega(I))_{I\in \mathcal D}, (v_I)_{I\in \mathcal D}\Big)$ satisfies $2/p$-{\bf DRCC } with constant $C>0$ then
\begin{equation}
\left\|\sum_{I\in \Lambda} \frac{H_I}{\Vert H_I\Vert_{p,\omega}}\right\|_{X^p(\omega)}\le C \left(\sum_{I\in \Lambda} \frac{\omega(I)}{v_I}\right)^{1/p}
\end{equation}
for all  finite family $\Lambda$ of dyadic intervals.
\end{lemma}
\begin{proof}
 Let $E = \cup_{I\in\Lambda}I$. As above
 $I(x)$ stands for the minimal dyadic interval in $\Lambda$ with regard to the inclusion relation that contains $x$. From (\ref{cc1}) we have that
  \begin{equation}\label{1}
  \sum_{I\in \Lambda}\omega(I)^{-2/p}\chi_I(x)\le Cv_{I(x)}^{-2/p}, \quad x\in E.\end{equation}
Now denote for each $I\in \Lambda$, $\tilde I=\{x\in E: I(x)=I\}$. Clearly $\tilde I\subseteq I$ and $E=\cup_{I\in \Lambda}\tilde I $. Hence applying (\ref{main}) and (\ref{1}) we obtain  \begin{eqnarray*}
\left\|\sum_{I\in \Lambda} \frac{H_I}{\Vert H_I\Vert_p}\right\|_{X^p(\omega)}&\le& C\left(\int_E
 v_{I(x)}^{-1}\omega(x) dx\right)^{1/p} = C\left(\int_{\cup_{I\in \Lambda} \tilde I}
 v_{I(x)}^{-1}\omega(x) dx\right)^{1/p} \\
&\le& C \Big(\sum_{I\in \Lambda} \int_{\tilde I}v_{I}^{-1}\omega(x)dx\Big)^{1/p}\le C \Big(\sum_{I\in \Lambda} v_I^{-1}\int_{I}\omega(x)dx\Big)^{1/p}\\
&=& C(\sum_{I\in \Lambda} \frac{\omega(I)}{v_I})^{1/p}.
\end{eqnarray*}
The proof is now complete.
\end{proof}

Combining Remark \ref{n} and Lemmas \ref{1c} and \ref{2p} we obtain the following corollary.
\begin{corollary} Let $1<p<\infty$,  $\omega$ be a weight satisfying the dyadic reverse doubling condition
then
\begin{equation} \label{basic}
\left\|\sum_{I\in \Lambda} \frac{H_I}{\Vert H_I\Vert_{p,\omega}}\right\|_{X^p(\omega)}\approx card(\Lambda)^{1/p}
\end{equation}
for all  finite family $\Lambda$ of dyadic intervals.
\end{corollary}

\begin{corollary} Let $1<p<\infty$,  $\omega$ be a weight and $(v_I)_{I\in \mathcal D}$ of positive real numbers. If
$\Big((\omega(I))_{I\in \mathcal D}, (v_I)_{I\in \mathcal D}\Big)$ satisfies $2/p'$-{\bf DRCC } with constant $C>0$ then
\begin{equation}\label{dem0}
\left(\sum_{I\in \Lambda} \frac{\omega(I)}{v_I}\right)^{1/p}\le C \Big(\max_{I\in \Lambda}\frac{\omega(I)}{v_I}\Big)\left\|\sum_{I\in \Lambda} \frac{H_I}{\Vert H_I\Vert_{p,\omega}}\right\|_{X^p(\omega)}
\end{equation}
for all  finite family $\Lambda$ of dyadic intervals.
\end{corollary}
\begin{proof}
Note that, using Lemma \ref{2p}, we have
\begin{eqnarray*}
\sum_{I\in \Lambda} \frac{\omega(I)}{v_I}&=& \int_0^1 (\sum_{I\in\Lambda} v_I^{-1}\chi_I(x))\omega(x)dx\\
&\le & \int_0^1 (\sum_{I\in\Lambda} \omega(I)^{-2/p}\chi_I)^{1/2}(\sum_{I\in\Lambda} v_I^{-2}\omega(I)^{2/p}\chi_I(x))^{1/2}\omega(x)dx\\
&\le & \Big(\int_0^1 (\sum_{I\in\Lambda} \omega(I)^{-2/p}\chi_I)^{p/2}\omega(x)dx\Big)^{1/p}\Big(\int_0^1(\sum_{I\in\Lambda} v_I^{-2}\omega(I)^{2/p}\chi_I(x))^{p'/2}\omega(x)dx\Big)^{1/p'}\\
&\le& \left\|\sum_{I\in \Lambda} \frac{H_I}{\Vert H_I\Vert_{p,\omega}}\right\|_{X^p(\omega)} \left\|\sum_{I\in \Lambda} \frac{\omega(I)}{v_I}\frac{H_I}{\Vert H_I\Vert_{p',\omega}}\right\|_{X^{p'}(\omega)}\\
&\le& \left(\max_{I\in \Lambda}\frac{\omega(I)}{v_I}\right)\left\|\sum_{I\in \Lambda} \frac{H_I}{\Vert H_I\Vert_{p,\omega}}\right\|_{X^p(\omega)}\left\|\sum_{I\in \Lambda} \frac{H_I}{\Vert H_I\Vert_{p',\omega}}\right\|_{X^{p'}(\omega)}\\
&\le& C\left(\max_{I\in \Lambda}\frac{\omega(I)}{v_I}\right)\left\|\sum_{I\in \Lambda} \frac{H_I}{\Vert H_I\Vert_{p,\omega}}\right\|_{X^p(\omega)}\left(\sum_{I\in \Lambda} \frac{\omega(I)}{v_I}\right)^{1/p'}.
\end{eqnarray*}
The result now follows.
\end{proof}
Taking into account that dyadic reverse Carleson condition of order $\alpha$ implies dyadic reverse Carleson condition of order $\beta$ for $\beta>\alpha$ we obtain the following fact.
\begin{corollary} Let $1<p<\infty$,  $\omega$ be a weight satisfying the dyadic reverse Carleson condition of order $\min\{2/p', 2/p\}$
then
\begin{equation} \label{basic}
\|\sum_{I\in \Lambda} \frac{H_I}{\Vert H_I\Vert_{p,\omega}}\|_{X^p(\omega)}\approx card(\Lambda)^{1/p}
\end{equation}
for all  finite family $\Lambda$ of dyadic intervals.
\end{corollary}

\begin{theorem} Let  $1<p<\infty$, $0<t\le 1$,  $(w_I)_{I\in \mathcal D}$ be a sequence of real numbers such that $$0<m_0=\inf_{I\in \mathcal D}w_I\le \sup_{I\in \mathcal D}w_I=M_0<\infty$$ and let $\omega$ be a weight satisfying the dyadic reverse Carleson condition of order $\min\{1, 2/p\}$ with constant $C>0$.
Then the Haar basis has the $(t, w_I)$-PCCG property in $X^p(\omega)$.
  \end{theorem}
  \begin{proof} 
Let $f\in X^p(\omega)$ and let $\Lambda^t_m$ be a set of $m$ dyadic intervals where $$\min_{I\in \Lambda^t_m}c_I(f,p, \omega)\ge t \max_{I'\notin \Lambda^t_m}c_{I'}(f,p, \omega) .$$  

 For  each $\alpha\in \R$, $(\varepsilon_n)\in \{\pm 1\}$ and $\Lambda'$ with $\sum_{J\in \Lambda'}w_J\le \sum_{I\in \Lambda_m^t}w_I$  we need to show that $\|f-P_{\Lambda^t_m}(f)\|_{X^p(\omega)}\le C(t) \|f-\alpha 1_{\varepsilon \Lambda'}\|_{X^p(w)}$ for some constant $C(t)>0$.
From triangular inequality
$$\Vert f-P_{\Lambda^t_m}(f)\Vert_{X^p(\omega)}  \leq \Vert P_{(\Lambda_m \cup \Lambda')^c}(f-\alpha 1_{\varepsilon \Lambda'})\Vert_{X^p(\omega)} + \Vert P_{\Lambda'\setminus \Lambda^t_m}(f)\Vert_{X^p(\omega)}$$
and the fact $\Vert P_{\Lambda}(f-\alpha 1_{\varepsilon B})\Vert_{X^p(\omega)}\le \Vert f-\alpha 1_{\varepsilon B}\Vert_{X^p(\omega)}$ for any $\Lambda$
 we only need to show that there exists $C>0$ such that
$$\Vert P_{\Lambda'\setminus \Lambda^t_m}(f)\Vert_{X^p(\omega)}\le  C \|f-\alpha 1_{\varepsilon \Lambda'}\|_{X^p(\omega)}.$$

Set $v_I=\frac{\omega(I)}{w_I}$ and observe that $\Big((\omega(I))_{I\in \mathcal D}, (v_I)_{I\in \mathcal D}\Big)$ satisfies $2/p$-{DRCC } with constant $M_0C$ and $\Big( (v_I)_{I\in \mathcal D}, \omega(I)_{I\in \mathcal D}\Big)$ satisfies $1$-{DRCC } with constant $C/m_0$.
Note that $\sum_{J\in \Lambda'}w_J\le \sum_{I\in \Lambda_m^t}w_I$  implies that
$$ \sum_{J\in \Lambda'\setminus \Lambda_m^t}\frac{\omega(J)}{v_{J}}\le \sum_{I\in \Lambda_m^t\setminus \Lambda'}\frac{\omega(I)}{v_{I}}  $$
and then, invoking Lemma \ref{2p} and Lemma \ref{1c}, we get the estimates
\begin{eqnarray*}
\Vert P_{\Lambda'\setminus \Lambda_m^t}(f)\Vert_{X^p(\omega)} &\leq& \Vert \underset{I\in \Lambda'\setminus A_m}{\max} c_I(f,p,\omega)1_{\Lambda'\setminus A_m}\Vert_{X^p(\omega)}\\
 &\le& C M_0 \underset{I\in \Lambda'\setminus \Lambda_m^t}{\max} c_I(f,p,\omega) (\sum_{J\in \Lambda'\setminus \Lambda_m^t}\frac{\omega(J)}{v_{J}})^{1/p}\\ &\leq&t^{-1} C M_0\underset{I\in \Lambda_m^t\setminus \Lambda'}{\min}c_I(f,p,\omega)(\sum_{I\in \Lambda_m^t\setminus \Lambda'}\frac{\omega(I)}{v_I})^{1/p} \\
 &\leq&  \frac{C^2 M_0}{tm_0}\Vert \underset{I\in  \Lambda_m^t\setminus \Lambda'}{\min} c_I(f,p,\omega)1_{\Lambda_m^t\setminus \Lambda'}\Vert_{X^p(\omega)}\\
&\leq& \frac{C^2 M_0}{tm_0}\Vert P_{\Lambda_m^t\setminus \Lambda'}(f)\Vert_{X^p(\omega)} \\
&=& \frac{C^2 M_0}{tm_0}\Vert P_{\Lambda_m^t\setminus \Lambda'}(f-\alpha 1_{\varepsilon B})\Vert_{X^p(\omega)} \\
&\leq &\frac{C^2 M_0}{tm_0} \|f-\alpha 1_{\varepsilon \Lambda'}\|_{X^p(\omega)}.
\end{eqnarray*}

This completes the proof with $C(t)= 1+\frac{C^2 M_0}{tm_0}$.
\end{proof}

\begin{corollary}
(i) If $\omega\in A_p^d$ then
 the Haar basis has the $t$-PCCG property (and hence is $t$-greedy) in $L^p(\omega)$  with $1<p<\infty$.

 (ii) The Haar basis has the $(t,w_I)$-PCCG property (and hence is $(t, w_I)$-greedy) in $L^p([0,1])$  for any sequence $(w_I)_{I\in\mathcal{D}}$ with $0<\inf w_I\le \sup w_I <\infty.$
  \end{corollary}
\noindent{\it \bf Acknowledgment:} The authors would like to thank to G. Garrig\'os and E. Hern\'andez for useful conversations during the elaboration of this paper.

\vspace*{-1.5cm}

\bibliographystyle{plain}

\begin{thebibliography}{1}
\bibitem{ABM}\textsc{H. A. Aimar, A. I. Bernardis, F.J. Martin-Reyes},\textit{Multiresolution approximation and wavelet bases of weighted $L^p$ spaces}, J. Fourier Anal. Appl. , \textbf{9} (2003), 497-510.	
\bibitem{BB}\textsc{P. M. Bern\'a, O. Blasco},\textit{Characterization of greedy bases in Banach spaces}, arXiv:
		1604.07260v1 [math.FA] 25 Apr 2016.	
		\bibitem{AA2}\textsc{F.Albiac, J.L. Ansorena},\textit{Characterization of 1-almost greedy bases}, arXiv:
		1506.03397v2 [math.FA] 17 Aug 2015.		
		\bibitem{AW}\textsc{F.Albiac, P.Wojtaszczyk},\textit{Characterization of 1-greedy bases}, J. Approx. Theory,\textbf{138} (2006), no.1, 65-86.
\bibitem{CDH}\textsc{A. Cohen, R.A. Devore, R. Hochmuth},\textit{Restricted nonlinear approximation}, Constr. Approx.,\textbf{16} (2000), no.1, 85-113.
\bibitem{DKOSS}\textsc{S.J. Dilworth,  D. Kutzarova, E. Odell, T. Schlumprecht, A. Zsak},\textit{Renorming spaces with greedy bases}, J. Approx. Theory \textbf{188} (2014),  39-56.	
		\bibitem{DKKT}\textsc{S.J. Dilworth, N.J. Kalton, D. Kutzarova, V.N. Temlyakov},\textit{The thresholding greedy algorithm, greedy bases, and duality}, Constr.Approx.\textbf{19} (2003), no.4, 575-597.	

\bibitem{GCRF}\textsc{J. Garc\'{\i}a-Cuerva, J.L. Rubio de Francia},\textit{Weighted norm inequalities and related topics}, North-Holland, Amsterdam, 1985.	



\bibitem{HV}\textsc{E. Hern\'andez, D. Vera},\textit{Restricted non-linear approximation in sequence spaces and applications to wavelet bases and interpolation}, Monatsh. Math.,\textbf{169} (2013), no.2, 187-217.
    \bibitem{I}\textsc{M. Izuki},\textit{The Haar wavelets and  the Haar scaling function in weighted $L^p$ spaces with $A_p^{dy,m}$ weights}, Hokkaido Math. J.,\textbf{36} (2007),  417-444.

    \bibitem{IS}\textsc{M. Izuki, Y. Sawano},\textit{The Haar wavelet characterization of weighted Herz spaces and greediness of the Haar wavelet basis}, J. Math. Anal. Appl.,\textbf{362} (2010),  140-155.

\bibitem{KPT}\textsc{G. Kerkyacharian, D. Picard, V.N. Temlyakov},\textit{Some inequalities for the tensor product of greedy bases and weight-greedy bases}, East J. Approx.\textbf{12} (2006), 103-118.
	
		\bibitem{VT}\textsc{S.V.Konyagin, V.N.Temlyakov},\textit{A remark on greedy approximation in Banach spaces}, East J. Approx. \textbf{5} (1999), 365-379.
		\bibitem{Tem} \textsc{V.N.Temlyakov},\textit{Greedy approximation}, Cambridge Monographs on Applied and Computational Mathematics, vol.20, Cambridge University Press, Cambridge, 2011.

\bibitem{T} \textsc{V.N.Temlyakov},\textit{The best $m$-term approximation and greedy algorithms}, Adv. Comp., \textbf{8} (1998), 249-265.
		\bibitem{Woj} \textsc{P.Wojtaszczyk},\textit{Greedy algorithm for general biorthogonal systems}, J.Approx.Theory \textbf{107} (2000), no.2, 293-314.
		\bibitem{Woj1} \textsc{P.Wojtaszczyk},\textit{Greedy type bases in Banach spaces}, Constructive theory of functions, 136-155, DARBA, Sofia, 2003.
	\end{thebibliography}

\end{document}